\documentclass[10pt,a4paper,reqno]{amsart}
\usepackage{amsfonts,amsthm,latexsym,amsmath,amssymb,amscd,amsmath,epsf, 
graphicx}

\newtheorem{tm}{Theorem}
\newtheorem{defi}[tm]{Definition}
\newtheorem{rem}[tm]{Remark}

\newtheorem{lm}[tm]{Lemma}
\newtheorem{ex}[tm]{Example}
\newtheorem{cor}[tm]{Corollary}
\newtheorem{prop}[tm]{Proposition}

\newtheorem{conj}[tm]{Conjecture}

\newcommand{\al}{\alpha}

\newcommand{\la}{\lambda}

\newcommand{\bC}{\mathbb C}
\newcommand{\bR}{\mathbb R}
\newcommand{\bN}{\mathbb N}
\newcommand{\bZ}{\mathbb Z}
\newcommand{\RP}{\mathbb{RP}}

\newcommand{\NA}{\mathcal{NA}}

\newcommand{\field}[1]{\mathbb{#1}}

\newcommand{\N}{\field{N}}

\newcommand{\sgn}{\operatorname{sgn}}

\newcommand{\Ln}{\operatorname{Ln}}

\begin{document}
\title[A tropical analog of Descartes' rule of signs]{A tropical analog of Descartes' rule of signs}

\author[J.~Forsg{\aa}rd]{Jens Forsg{\aa}rd}
\noindent 
\address{Department of Mathematics,
Texas A\&M University, College Station, TX 77843.}
\email{jensf@math.tamu.edu}

%\author[V.~Kostov]{Vladimir P. Kostov }
%\noindent 
%\address{Universit\'e de Nice, Laboratoire de Math\'ematiques, 
%Parc Valrose, 06108 Nice Cedex 2, France}
%             \email{kostov@math.unice.fr}

\author[D.~Novikov]{Dmitry Novikov}
\noindent 
\address{Department of Mathematics, Weizmann Institute of Science, Rehovot, Israel}
\email{dmitry.novikov@weizmann.ac.il}
\thanks{Supported by the Minerva foundation with funding from the Federal 
          	German Ministry for Education and Research.}

\author[B.~Shapiro]{Boris Shapiro}
\noindent 
\address{Department of Mathematics, Stockholm University, SE-106 91
Stockholm,
         Sweden}
\email{shapiro@math.su.se}

\subjclass[2010] {Primary 26C10, Secondary 30C15}

\keywords{ tropicalization, real and tropical roots, Descartes' rule of signs}

\begin{abstract}
We prove that for any degree $d$, there exist  (families of) finite  sequences $\{\la_{k,d}\}_{0\le k\le d}$ of positive numbers such that, for any real polynomial $P$ of degree $d$, the number of its real roots is less than or equal to the number of the so-called essential tropical roots of the polynomial obtained from $P$ by multiplication of its coefficients by $\la_{0,d},\la_{1,d},\dots , \la_{d,d}$ respectively. In particular, 
for any real univariate polynomial $P(x)$ of degree $d$ a with non-vanishing constant term, we conjecture that one can take $\la_{k,d}=e^{-k^2},\,k=0,\dots, d $. The latter claim can be thought of  as a tropical generalization of Descartes's rule of signs. We settle this conjecture up to degree $4$ as well as a weaker statement for arbitrary real polynomials.  Additionally we  describe an application of the latter conjecture  to the classical Karlin  problem on zero-diminishing sequences. 
\end{abstract}

\date{}
\maketitle

\section{Introduction}

The famous Descartes' rule of signs claims that the number of positive roots 
of a real univariate polynomial does not exceed the number of sign changes 
in its sequence of coefficients. In what follows, among other things,  we suggest a conceptually new 
conjectural upper bound on the number of real roots of real univariate polynomial applicable in 
the situation when Descartes' rule of signs gives a trivial restriction. 

\medskip
Recall from the literature that a sequence $ \lambda = \{\lambda_k\}_{k=0}^\infty$
of real numbers is called a \emph {multiplier sequence {\rm(}of the first kind{\rm)}} if the
diagonal operator $T_\lambda\colon \bR[x]\rightarrow \bR[x]$ 
defined by $x^k \mapsto \lambda_k x^k$, for $k=0, 1, \dots$,
and extended to $\bR[x]$ by linearity, 
preserves the set of real-rooted polynomials, see e.g.,  \cite{CC04}.
To formulate  our results, 
we need to introduce tropical analogs of
multiplier sequences.  
The following notion is borrowed from the classical Wiman--Valiron theory,
see e.g., \cite{Hay74}.
A non-negative integer $k$ is called a \emph{central index} of a  
polynomial
\[
P(x) = \sum_{i=0}^d a_i x^i
\]
if there exists a real number $x_k\geq 0$ such that
\begin{equation}
\label{eqn:Lopsidedness}
 |a_k|x_k^k \geq \sum_{i\neq k}|a_i|x_k^i.
\end{equation}
Condition \eqref{eqn:Lopsidedness} has also reappeared in the context of
amoebas, see, e.g., \cite{Rul03}.
%In amoeba theory, one says that $f$ is \emph{lopsided}
%at $x_m$ with respect to the monomial of exponent $m$, see \cite{Pur08}.

To relate property \eqref{eqn:Lopsidedness} to real-rootedness of univariate polynomials, we recall 
that a real-rooted polynomial $P$ is called  \emph{sign-independently real-rooted} if each 
polynomial obtained by an arbitrary sign change of the coefficients of $P(x)$
is real-rooted as well, see \cite{PRS11}. One can easily show the following  statement.  

\begin{prop}
\label{pro:SignIndependently}
 A real polynomial $P$ of degree $d$ is sign-independently real-rooted
 if and only if every integer $k=0, \dots, d$ is a central index of $P$.
\end{prop}

To proceed, we will need the following similar notion.
A non-negative integer $k$ is said to be a \emph{tropical index} of $P$ if there exists 
a number $x_k\geq 0$ such that
\begin{equation}
\label{eqn:Tropical}
 |a_k|x_k^k \geq \max_{i\neq k}\,|a_i|x_k^i.
\end{equation}

Notice that \eqref{eqn:Tropical} is an analog of 
\eqref{eqn:Lopsidedness} if the right-hand side of \eqref{eqn:Lopsidedness} is 
interpreted as a tropical sum.
We will say that a polynomial $P$ of degree $d$
is \emph{tropically real-rooted}
 if each integer $k=0, \dots, d$ is a tropical index of $f$.

%We say that a 
%sequence $ \lambda$ of real numbers is a \emph{tropical  multiplier sequence}, if the diagonal operator $T_\la\colon \bR[x]\rightarrow \bR[x]$ (see above) preserves the set of 
%tropical indices of every real polynomial $P$. As we will see later, this  condition is equivalent to the fact 
%that $T_\lambda$ preserves the set of central indices of every real polynomial $P$.

\medskip
By the (standard) \emph{tropicalization} of a real polynomial $P(x)=\sum_{i=0}^d a_i x^i$ we mean the tropical polynomial
 given by:
\begin{equation}
\label{eqn:tropicalization}
tr_P(\xi) =\max_{0\le i\le d} (i\xi +\ln |a_i|),\;\xi\in \bR.
\end{equation}
(In the literature the function $tr_P(\xi)$ is also referred to  as the \emph{Archimedean tropical
polynomial} associated to $P$.) 
If $a_i = 0$, then the corresponding term in $tr_P(\xi)$ should be interpreted
as  $-\infty$, and thus it can be ignored when taking the maximum.
\begin{rem}{\rm 
One can describe $tr_P(\xi)$ as follows. Define the set of points on $(u,v)$-plane corresponding to the monomials of $P$ as $A_P=\{(k,\log |a_k|), k=0,...,d\}$.
Let $A_P(u)$ be a piecewise linear continuous function on $[0,d]$, linear on intervals $[k,k+1]$ and such that $A_P(k)=\log |a_k|$ for $k=0,...,d$.
Denote by $\widetilde A_P(u)$ the least concave majorant of $A_P(u)$ on $[0,d]$, and let 
$$
\NA_P=\{v\le \widetilde A_P(u), 0\le u\le d\}
$$ 
be the Archimedian Newton polytope of $P$. Then $k$ is a tropical index of $P$ if and only if $(k, \log |a_k|)$ is a boundary point of $\NA_P$. 
Then $tr_P(\xi)=\max_{p\in \NA_P} (\xi, 1)\cdot p$, i.e., $tr_P(\xi)$ is the support function of $\NA_P$. Alternatively, $tr_P(\xi)$ is the Legendre transform of $-\widetilde A_P(u)$.}
\end{rem}
\medskip
Any corner of the graph of $tr_P(\xi)$, i.e., a value of $\xi$ at which its slope 
changes, is called a {\it tropical root} of $tr_P(\xi)$. We define the (tropical) \emph{multiplicity} of a tropical root $\zeta$ of $tr_P$ to be 
one less than 
the number of terms of \eqref{eqn:tropicalization} for which the maximum in the right-hand side of \eqref{eqn:tropicalization} 
is attained at $\zeta$. (Notice that this definition differs from the standard definition of 
root multiplicity in tropical geometry.
This illustrates our focus on real  rather
than complex-valued polynomials.) 
With our definition of tropical root multiplicity, the number of tropical roots of $tr_P(\xi)$ counted
with multiplicities is one less than the number of tropical indices of $P$.
In particular, the number of tropical roots of $tr_P(\xi)$ is at most by 
one less than the number of monomials of $P$, which is analogous to  the fact that the  number of 
real roots of $P$ is at most one less than its number of monomials.

\medskip

We will now  define positive and negative tropical roots of $P$ using the signs of its coefficients. 
%Let $\sigma := \sigma(P) = (\sgn(a_0), \dots, \sgn(a_d))$ be the sign vector of $P$.
%Here, we use the convention that $\sgn(0) = 0$, thus $\sigma$ has components
%in the set $\{-1,0,1\}$.
Let $k_0 \le k_1 \le \dots \le k_m$ be the tropical indices of $P$.
Consider two sequences  $\{\sgn(a_{k_i})\}_{0\le i \le m}$ and  $\{\sgn((-1)^{k_i} a_{k_i})\}_{0\le i \le m}$. 

 Consider two consecutive tropical indices $k_{i-1}$ and $k_{i}$ of the polynomial $P$; 
 to this pair we associate the tropical root $\xi_i = \ln(a_{i-1}/a_{i})$ of $tr_P(\xi)$.
 If the difference $k_{i+1} - k_{i}$ is odd, then the pair $(k_{i-1}, k_i)$ contributes 
 a sign alternation in exactly one of the above sequences. In this case, we will say that $\xi_i$
 is a {\it positive} (respectively {\it negative}) {\it essential tropical root} of $P$. 
%We call by  the number of \emph{negative tropical roots of $P$} 
 %the number of sign alternations in the sequence $\{\sgn((-1)^{k_i} a_{k_i})\}_{0\le i \le m}$,
 %and by the number of \emph{positive tropical roots of $P$}
 %the number of sign alternations in the sequence $\{\sgn(a_{k_i})\}_{0\le i \le m}$.
 If the difference $k_{i+1} - k_i$ is even, then either the pair $(k_{i-1},k_i)$ does not
 contribute  a sign alternation in any of the above sequences, or it contributes 
 a sign alternation in both. In the former case we will say that $\xi_i$ is a {\it non-essential
 tropical root }of $P$, and in the latter case we will say that $\xi_i$ is a {\it positive-negative 
 essential tropical root} of $P$.  By  the number of \emph{positive essential tropical roots of $P$} we mean the sum of the number of positive and positive-negative tropical roots of $P$. 
 Analogously, by  the number of \emph{negative essential tropical roots of $P$} we mean the sum of the number of negative and positive-negative tropical roots of $P$. Finally  by  the  total number of \emph{essential tropical roots of $P$} 
we call  the sum of the above two numbers. 
%(The multiplicities of  positive-negative essential roots are counted  twice in this total number.)   
 
 It is easy to see that the number of essential tropical roots of $P$ is at most $d$.

 \begin{ex}
{\rm Consider  $P_1(x) = 1 + x^2$. The tropical indices of $P_1$ are  $k_0 = 0$ and $k_1 = 2$.
 As $\ln|a_1| = \ln|0| = -\infty$, the polynomial $P_1$ has (with our definition of multiplicity) exactly one
simple tropical root. To count the number of positive and negative tropical roots of $P_1$ we need to  count the
 number of sign alternations in the sequences $\{1,1\}$ and $\{1, (-1)^2\} = \{1,1\}$ respectively.
 That is, the number of essential tropical roots of $P$ is equal to $0$.
 
 Consider now the polynomial $P_2(x) = 1 -  x^2$. Similarly to  $P_1$, the polynomial $P_2$ has
 one tropical root. However, to count the number of positive and negative tropical roots of $P_2$ we
 count the number of sign alternations in the sequences $\{1,-1\}$ and $\{1, -(-1)^2\} = \{1,-1\}$ respectively.
 That is, the number of essential tropical roots of $P_2$ is equal to $2$.}
 \end{ex}

\medskip 
As the definitions of the central and the tropical indices only depend
on the modulis $|a_i|$, for $i=0, \dots, d$, 
they  immediately extend to complex-valued polynomials.
However, below we restrict ourselves only to real polynomials
%for simplicity, we will henceforth assume that $a_k \geq 0$ for all $k$, 
and positive sequences $ \lambda$. 

\medskip
A sequence $\lambda=\{\la_k\}_{k=0}^\infty$ is called \emph{log-concave} if 
$\lambda_k^2 \geq \lambda_{k-1}\lambda_{k+1}$ for all $k$.
In \cite{PRS11} using discriminant amoebas, it is proven  
that the diagonal operator $T_ \lambda\colon \bR[x]\rightarrow \bR[x]$ %associated to the sequence $\lambda$ 
preserves the set of sign-independently
real-rooted polynomials if and only if $ \lambda$ is log-concave.
For this reason, log-concave sequences were called 
\emph{multiplier sequences of the third kind} in \emph{loc.~cit.} 
We prefer to refer to log-concave sequences $ \lambda$  as \emph{tropical multiplier sequences}. 

\begin{defi}
 {\rm A positive sequence $ \lambda= \{\lambda_k\}_{k=0}^\infty$ is said to be a \emph{ tropical {\rm(}resp.\ central{\rm)} index preserver} 
 if for each polynomial $P$ the set of tropical (resp.\ central) indices
 of $P$ is a subset of the set of tropical (resp.\ central) indices
 of the polynomial $T_ \lambda[P]$.}
\end{defi}

Our first result is as follows.
\begin{tm}
\label{thm:TropicalIndexPreserver}
 For  positive sequences $\lambda,$ the following three conditions are equivalent: 
\begin{enumerate}
\item $\lambda$ is log-concave, i.e. $\la$ is a tropical mutliplier sequence;
\item $\lambda$ is a tropical index preserver;
\item $\lambda$ is a central index preserver.
%\item $\lambda$ preserves the set of sign-independently
%real-rooted polynomials. 
\end{enumerate}
\end{tm}

%Out main theorems are the following characterizations of
%tropical and central index preservers. 
In particular, Theorem~\ref{thm:TropicalIndexPreserver} provides an alternative (and elementary) way  to settle  \cite[Theorem 1]{PRS11}
as requested in %\cite[Problem  2]{PRS11}.
Problem 2 of \emph{loc.\ cit.}
\begin{cor}
\label{cor:SignIndependently}
A positive sequence $ \lambda$ preserves the set of sign-independently 
real-rooted polynomials if and only if it is log-concave.
\end{cor}

In what follows, we will need a slightly more general definition of
a tropicalization of $P$.
Given an arbitrary triangular sequence  
$ \lambda=\{\la_{k,j}\}_{0\le k\le j,\; j\in \bN} $ of positive 
numbers, and a univariate polynomial $P(x)=\sum_{i=0}^d a_ix^i$  of any degree $d$,
%$P(x)=a_dx^d+\dots + a_0,\; a_d\neq 0$,
%with non-negative coefficients, 
we define its {\it $\la$-tropicalization}  as 
\begin{equation}\label{eq:Latrop}
tr^\la_P(\xi)=\max_{0\le k\le d}( k\xi +\ln |a_k|+\ln \la_{k,d}),\; \xi\in \bR.
\end{equation}

\begin{rem}{\rm 
	Here is another description of $tr^\la_P(\xi)$. Let $\Theta_d(u)$ be a continuous piecewise linear function on $[0,d]$, linear on intervals $[k,k+1]$ for $k=0,...,d-1$ and such that $\Theta_d(k)=\log \lambda_{k,d}$ for $k=0,...,d$. Define $A^\lambda_P(u)$ as the least concave majorant of  $A_P(u)+\Theta_d(u)$. Then $tr^\lambda_P$ is the Legendre transform of $-A^\lambda_P(u)$, see Remark~2.} 
\end{rem}
%The same standard convention that if some $a_j$ vanishes then we skip the 
%corresponding linear form in \eqref{eq:Latrop}  is valid.  

\begin{defi}%[Main Definition]
{\rm A finite sequence $\{\la_{k,d}\}_{0\le k\le d},$ of positive numbers is 
called %{\it (positively) zero-non-diminishing in degree $d$,} 
a \emph{ degree $d$ {\rm (}positive{\rm)} real-to-tropical root preserver}
if for any polynomial $P$ of degree $d$ (with positive coefficients), 
the number of essential tropical roots of %the tropical polynomial 
\eqref{eq:Latrop} is greater than or equal to the number of non-zero real roots of $P$. 
A triangular sequence $ \lambda=\{\la_{k,j}\}_{0\le k\le j,\; j\in \bN}$ is called 
a \emph{{\rm (}positive{\rm)} real-to-tropical root preserver}
%{\it (positively) zero-non-diminishing}, if, for every $d$, 
if for each $d$ its finite subsequence $\{\la_{k,d}\}_{0\le k\le d},$ is  
a \rm  degree $d$ (}positive{\rm) real-to-tropical root preserver.
%(positively) zero-non-diminishing in degree $d.$}
}
\end{defi}

We recall  that the \emph{recession cone}
of a set $X\subset \bR^{d+1}$ is the largest pointed (i.e. including the origin) cone  $C\subseteq \bR^{d+1}$ 
such that if $x\in X$ then $x+c\in X$ for all $c\in C$. Our main result is as follows. 

\begin{tm}\label{th:main} The set $\Lambda_d\subset \bR_+^{d+1}$ 
{\rm(}respectively $\Lambda_d^+\subset \bR_+^{d+1}${\rm)} of all
  degree $d$ {\rm(}positive{\rm)} real-to-tropical root preservers $\{\la_{k,d}\}_{0\le k\le d}$  %sequences  
 is a nonempty closed full-dimensional subset 
of\/ $\bR_+^{d+1}$. 
Moreover, the recession cone of its logarithmic image $\Ln(\Lambda_d)$
{\rm(}respectively $\Ln(\Lambda^+_d)${\rm)} 
coincides with the cone of all  concave sequences of length $d+1$. {\rm(}Here  for any $\Omega\subset \bR_+^k$, by $\Ln(\Omega)$ we mean the set in $\bR^k$ obtained by taking natural logarithms of points from $\Omega$ coordinatewisely.\rm{)}  
%(That is, if $ \lambda = \{\lambda_{k,d}\}_{0\le k\le d} \in \Lambda_d$,
%then so does the sequence $\{\lambda_{k,d}\lambda^*_k\}_{0\le k\le d}$
%for any log-concave sequence $\lambda^*=\{\lambda^*_k\}_{0\le k\le d}$.)
\end{tm}

%\begin{tm}\label{th:main2} The set $\Lambda_d^+\subset \bR_+^{d+1}$ of all
%positively zero-non-diminishing in degree $d$ sequences  
%$\{\la_{k,d}\}_{0\le k\le d}$ is a nonempty closed full-dimensional subset 
%of $\bR_+^{d+1}$. Moreover, the recession cone of its logarithmic image $\Ln(\Lambda_d)$
%is the cone of all concave sequences.
%\end{tm}

Theorem~\ref{th:main} shows that there exist large families  of real-to-tropical root preservers
%zero-non-diminishing sequences 
in each degree, and therefore large families  of real-to-tropical root preserving 
%zero-diminishing 
triangular sequences.

\medskip

First we show that, if $\la=\{\lambda_{k,d}\}_{0\le k\le d}$ is sufficiently log-concave, then $\lambda$ is a degree $d$ real-to-tropical root preserver:
\begin{tm}\label{thm:example}
	Assume that a sequence $\la=\{\lambda_{k,d}\}_{0\le k\le d}$ of positive numbers satisfies the condition: 
	\begin{equation}\label{eq:example}
	\log \frac{\lambda^2_{k,d}} {\lambda_{k-1,d}\lambda_{k+1,d}}>2\Delta_d:=\frac{d^2}4\log 36d+(d+1)\log d+\log 4, \quad 1\le k \le d-1.
	\end{equation}
	Then, for any real polynomial $P$, the number of positive (negative) tropical roots of $tr^\la_P$ is greater than or equal  to the number of positive (negative) roots of $P$. In particular, $\la$ is a real-to-tropical root preserver. 
\end{tm}

Next we show that to be a real-to-tropical root preserver, the sequence $\la=\{\lambda_{k,d}\}_{0\le k\le d}$ should be sufficiently log-concave. 
\begin{tm}\label{thm:counterexample}
There exists  $c>0$ with the following property.	Assume that for some $k<d-100$ 
\begin{equation}\label{eq:counterexample}
\log \frac{\lambda^2_{j,d}} {\lambda_{j-1,d}\lambda_{j+1,d}}<2c, \quad {j=k,...,k+100}.
\end{equation}
  Then there exists a polynomial $P$ of degree $d$ with positive coefficients such that $tr^\lambda_P$ has three tropical roots, and $P$ has four negative roots. In particular, $\{\lambda_{k,d}\}_{0\le k\le d}$ cannot be a degree $d$ (positive) real-to-tropical root preserver.
\end{tm}

\medskip
In this direction, we present the following tantalizing conjecture. Consider the sequence $\la^\dagger$ given by
  $$\la^\dagger_k:=e^{-k^2},\;k=0,1,\dots.$$
 we will denote by 
${tr}_P^\dagger(\xi)$ the corresponding tropical polynomial 
associated to any real polynomial $P$, i.e.
\begin{equation}\label{eq:Latropdagge}
tr^\dagger_P(\xi)=\max_{0\le k\le d}( k\xi +\ln |a_k|-k^2),\; \xi\in \bR.
\end{equation}

\begin{conj}[Conjectural tropical analog of Descartes' rule of signs]\label{conj:main} For any real univariate polynomial $P(x)$, the number of its positive (negative) roots does not
exceed the number of positive (negative) essential tropical roots of $tr^\dagger_P(\xi)$.
%\footnote{Although this conjecture seems as the first glance unmotivated, but...} 
\end{conj}

We have the following  partial result supporting Conjecture~\ref{conj:main}. 

%\begin{prop}\label{pr:hyp}
%Conjecture~\ref{conj:main} holds for all polynomials with positive coefficients and all simple real roots. 
%\end{prop}

\begin{prop}\label{pr:6}
Conjecture~\ref{conj:main} holds for $d\le 4$.
\end{prop}

Besides the fact that Conjecture~\ref{conj:main} looks quite appealing,  it might also shed  light on possible extensions of the classical Newton  inequalities for polynomials with a non-maximal number of real roots and positive coefficients. Additionally,  (if settled) it  also gives 
interesting consequences in the classical Karlin problem on zero-diminishining sequences, see \cite{Ka} and \S~\ref{sec:appl}.
%, see Section~\ref{sec:appl}. 

\medskip\noindent
{\em Acknowledgements.} The  third author wants to thank Professor V.~l.~Kostov of  Universit\'e de Nice for discussions.  

%%%%%%%%%%
\section{Introductory results and Theorem~\ref{th:main}}
%%%%%%%%%%

We will begin with the following statement. Given a sequence $\la=\{\la_k\}_{k=0}^\infty$, define its {\it symbol}  as the formal series $S_\la(x):=\sum_{k=0}^\infty \la_kx^k$.  Define its $d$-th truncation as 
$S_\la^{\{d\}}(x):=\sum_{k=0}^d \la_k x^k.$

\begin{lm}
\label{lem:LogConcaveTropical}
A positive sequence $\lambda$ is log-concave if and only if, for each $d,$ the 
$d$-th truncation $S_\la^{\{d\}}(x)$  
is a tropically real-rooted polynomial.
\end{lm}

\begin{proof}[Proof\/ of\/ Lemma \ref{lem:LogConcaveTropical}]
Assume first that $\lambda$ is log-concave. For each $m \geq 1,$ set 
% \[
% x_m = \lambda^{\frac 1 2}_{m-1}\lambda_{m+1}^{-\frac 1 2}.
% \]
$x_m := \sqrt{\lambda_{m-1}/\lambda_{m+1}}$.
Then,
\[
\frac{x_{m+1}}{x_m} 
%= \left(\frac{\lambda_{m}\lambda_{m+1}}{\lambda_{m+2}\lambda_{m-1}}\right)^{\frac 1 2}
%= \left(\frac{\lambda^2_{m}}{\lambda_{m-1}\lambda_{m+1}}\frac{\lambda^2_{m+1}}{\lambda_m\lambda_{m+2}}\right)^{\frac 1 2} \geq 1,
= \frac{\lambda_{m}}{\sqrt{\lambda_{m-1}\lambda_{m+1}}}\frac{\lambda_{m+1}}{\sqrt{\lambda_m\lambda_{m+2}}} \geq 1,
\]
so that $\{x_m\}_{m=1}^\infty$ is a non-decreasing sequence of positive real
numbers. Further more,
% \[
% \frac{\lambda_m x_m^m}{\lambda_{m-1} x_m^{m-1}} = \frac{\lambda_m}{\lambda_{m-1}^{\frac 1 2}\lambda_{m+1}^{\frac 1 2}} \geq 1
% \]
% and
% \[
% \frac{\lambda_{m} x_m^m}{\lambda_{m+1}x_m^{m+1}} = \frac{\lambda_{m}}{\lambda^{\frac 1 2}_{m-1}\lambda_{m+1}^{\frac 1 2}} \geq 1.
% \]
\[
\frac{\lambda_m x_m^m}{\lambda_{m-1} x_m^{m-1}} = 
\frac{\lambda_m x_m^m}{\lambda_{m+1} x_m^{m+1}} = 
\frac{\lambda_m}{\sqrt{\lambda_{m-1}\lambda_{m+1}}} \geq 1.
\]
Since both binomials $\lambda_k x^k - \lambda_{k+1}x^{k+1}$ and 
$\lambda_k x^k - \lambda_{k-1}x^{k-1}$ have exactly 
one positive real root, we conclude that
$\lambda_k x_m^k\geq \lambda_{k+1}x_m^{k+1}$ if $k\geq m$ and that
$\lambda_k x_m^k\geq \lambda_{k-1}x_m^{k-1}$ if $k\leq m$.
Hence,
\[
\lambda_m x_m^m \geq \max_{k\neq m}\, \lambda_k x_m^k.
\]

For the converse, assume that $\lambda$ is not log-concave. That is, there exists
an index $m$ for which $\lambda_m^2 < \lambda_{m-1}\lambda_{m+1}$.
Then, for $x\geq 0$,
\[
\lambda_m x^m < \sqrt{\lambda_{m-1}x^{m-1}\, \lambda_{m+1}x^{m+1}}
\leq \max \left(\lambda_{m-1}x^{m-1}, \lambda_{m+1}x^{m+1}\right).
\]
In particular, $m$ is not a tropical index of $S_\lambda(x)$.
\end{proof}

\begin{proof}[Proof\/ of\/ Theorem \ref{thm:TropicalIndexPreserver}]
Let us first prove that a sequence $\lambda$ is log-concave
if and only if it is a tropical index preserver.
Assume first that $\lambda$ is log-concave.
Let $m$ be a tropical index of $P$, and let $x_m\geq 0$ be such that
\[
a_m x_m^m \geq \max_{k\neq m}\, a_k x_m^k.
\]
By Lemma \ref{lem:LogConcaveTropical} we can find  $\zeta_m$ such that
\[
\lambda_m\zeta_m^m \geq \max_{k\neq m}\, \lambda_k\zeta_m^k.
\]
Then
\[
\lambda_ma_m(z_m \zeta_m)^m =\lambda_mx_m^m\,a_m \zeta_m^m \geq 
\lambda_kx_m^k\,a_k \zeta_m^k
\]
for all $k$. Hence, $m$ is a tropical index of $T_\lambda[P]$.
For the converse, it suffices to consider the sequence of polynomials $1+x+\dots + x^d$,
which are tropically real-rooted for all $d$, and 
use Lemma \ref{lem:LogConcaveTropical}.

Let us now prove that $\lambda$ is log-concave if and only if
it is a central index preserver.
Assume first that $\lambda$ is log-concave, and
let $\zeta_m$ be as in the proof of Theorem \ref{thm:TropicalIndexPreserver}.
Let $m$ be a central index of $P$, and let $x_m$ be such that
\[
a_mx_m^m \geq \sum_{k\neq m} a_k x_m^k.
\]
Then,
\[
\lambda_m a_m (x_m\zeta_m)^m \geq \sum_{k\neq m}\lambda_m \zeta_m^m a_k x_m^k
\geq \sum_{k\neq m}\lambda_k \zeta_m^k a_k x_m^k,
\]
implying that $m$ is a central index of $T_\lambda[P]$.
For the converse, assume that $\lambda_m^2 < \lambda_{m-1}\lambda_{m+1}$,
and consider the action of $T_\lambda$ on the trinomial $x^{m-1} + 2 x^m + x^{m+1}$.
\end{proof}

Using Lemma \ref{lem:LogConcaveTropical}, we can rephrase
Theorem \ref{thm:TropicalIndexPreserver} %and \ref{thm:CentralIndexPreserver}
in a manner similar to the classical result of 
P{\'o}lya and Schur, see \cite{PS14}. Given a  sequence $\lambda$ of real numbers, we say that its  symbol $S_\la(x)$ is {\it tropically real-rooted} if for each $d=0,1,\dots$, the $d$-th truncation 
$S_\la^{\{d\}}(x)$ %:=\sum_{k=0}^d \la_k x^k$ 
is tropically real-rooted.

\begin{cor}
A positive sequence $\lambda$
is a central index and tropical index preserver if and
only if its \emph{symbol} $S_\lambda(x)$ is
tropically real-rooted.\qed   %Du gl\"omde att \"andra $P_\la$!
\end{cor}

\begin{proof}[Proof of Proposition \ref{pro:SignIndependently}]
To prove the \emph{only if}-part, consider the polynomial
\[
Q(x) = |a_m|x^m - \sum_{k\neq m} |a_k| x^k,
\]
for some $1\le m\le d-1$.
Notice that $Q$ is is obtained from $P$ by flipping signs of the coefficients
and hence, by assumption, $Q$ is real-rooted. In particular, $Q$ has exactly two positive
roots (counted with multiplicity). Let $x_m$ be the mean value of the positive roots
of $Q$. Then,
\[
 |a_m|x_m^m - \sum_{k\neq m} |a_k| x_m^k \ge 0,
\]
with equality if and only if $Q$ has a positive root of multiplicity two.
In particular, $m$ is a central index of $P$.

For the \emph{if}-part, choose arbitrary signs of the coefficients of $P$.
We note that condition \eqref{eqn:Lopsidedness} 
implies that
\[
\sgn(P(x_m)) = \sgn(a_mx_m^m) = \sgn(a_m),
\]
for $x>0$. Using additionally Descartes' rule of signs, we conclude that
the number of positive roots of $P$ is equal to the number of sign
changes in the sequence $\{a_k\}_{0\le k\le d}$. Similarly, the number 
of negative roots of $P$ is equal to the number of sign changes in
the sequence $\{(-1)^ka_k\}_{0\le k \le d}$. 
As $a_k\neq 0$ for each $k$, these two numbers
sum up to $d$, implying that $P(x)$ is real-rooted.
Since the signs of the coefficients were chosen arbitrary,
we are done.
\end{proof}

\begin{proof}[Proof\/ of\/ Corollary \ref{cor:SignIndependently}]
It follows from Proposition \ref{pro:SignIndependently} that a positive sequence $\la$ preserves the set of sign-independently
real-rooted polynomials if and only if it preserves central indices. Additionally, 
it follows from Theorem \ref{thm:TropicalIndexPreserver} that a positive sequence $\la$ preserves central indices if and only if it is log-concave.
\end{proof}

\begin{proof}[Proof of Theorem~\ref{th:main}]
As we are only concerned with the number of (real) roots of the polynomial
$P$, we can consider $P$ up to a non-vanishing scalar, i.e., we identify $P$ with its coefficient vector
$(a_0:\ldots:a_d) \in \RP^d$.

Let us first show that the set $\Lambda_d$ is nonempty.
Let $\lambda = \{\lambda_k\}_{0\le k\le d}$ be a finite positive strictly log-concave sequence.
By Lemma~\ref{lem:LogConcaveTropical} we have that $S_\lambda^{\{d\}}(x)$ is
tropically real-rooted. Moreover  it follows from the proof of Lemma~\ref{lem:LogConcaveTropical}  and the strict log-concavity that all the tropical roots of $S_\lambda^{\{d\}}(x)$ are  of multiplicity one.
%Consequently, the
%tropical polynomial $tr_{S_\lambda^{\{d\}}}(\xi)$ has $d$ distinct tropical roots,
%implying in particular that it has the maximal number of distinct essential tropical roots
%for all choices of signs  of the coefficients 
%(including also the case of some vanishing coefficients).

Firstly, for each $P\in \RP^d$, we claim that there exists a positive number
$s = s(P)$ such that $tr_P^{\lambda^s}(\xi)$ has at least as
many distinct negative tropical roots as the number of negative roots of $P$.
Here, %$\sigma = \sigma(P)$ is the sign sequence of $P$ and
$\lambda^s$ denotes the sequence $\{\lambda_k^s\}_{0\le k\le d}$.
To prove this, notice first that, by using the change of variables
$\xi \mapsto s\xi$,   the number of negative tropical roots of
\[
 tr_P^{\lambda^s}(\xi) = \max_{0\le k\le d} (k\xi + \ln|a_k| + s \ln\lambda_k)
\]
is equal to the number of negative tropical roots of the tropical polynomial
\[
 \max_{0\le k\le d} \left(s\left(k\xi + \frac{\ln|a_k|}{s} + \ln \lambda_k\right)\right),
\]
the latter being equal to the Descartes' bound on the maximal number of
negative roots of $P,$ for $s$ sufficiently big. Indeed, 
for all $a_k \neq 0$ the term $(\ln|a_k|)/s$ tends to $0$ as $s \rightarrow \infty$.

Secondly, we claim that $s = s(P)$ can be chosen in such a way that
there exists a neighborhood $N(P)\subset \RP^d$ of $P$
such that for each $Q\in N(P)$ the number of
negative essential tropical roots of $tr_Q^{\lambda^s}$
is not less than the number of negative roots of $Q$.
Consider first the case  $a_0 \neq 0$.
Then, there is a neighborhood $N_1(P)$ of $P$
such that the number of negative roots of $Q\in N_1(P)$ is
at most equal to the number of negative roots of $P$. Since all negative tropical roots of $tr_P^{\lambda^s}$ are  distinct,
there is a neighborhood $N_2(P)$ such that the number of negative
tropical roots of $tr_P^{\lambda^s}$ is equal to the number of
negative tropical roots of $tr_Q^{\lambda^s}$ for all $Q\in N_2(P)$.
(If $P$ has some vanishing coefficients, then $N_2(P)$
can be chosen so that the corresponding indices are not tropical indices
of $Q$ for any $Q\in N_2(P)$.)
In this case we can take $N(P) = N_1(P)\cap N_2(P)$. 
Complementarily, consider the case $a_0 = 0$.
For each polynomial $Q$, let $Q'$ denote the polynomial obtained by removing the
constant term of $Q$. Using an inductive argument, we can choose a neighborhood $N(P)$ of
$P$ such that, for each $Q\in N(P)$, the number of
 negative tropical roots of $tr_{Q'}^{\lambda^s}$
is not less than the number of negative roots of $Q'$.
Notice that for the first non-zero coefficient $a_k$ of $P$,  
$k$ is a tropical index of $P$. If $(-1)^ka_k$ is positive, then the number
of negative real roots of $P$ increases by one if $a_0$ is 
perturbed by a small negative number, and similarly the number of 
negative tropical roots is increased by one,
and vice versa.

Finally, to see that $\Lambda_d$ is nonempty, we note that $\RP^d$ is
compact. Therefore, the open covering $\cup_{P\in \RP^d} N(P)$ 
of $\RP^d$ has
a finite subcovering $\RP^d\subset N(P_1)\cup\dots \cup N(P_M)$.
Let $s^* = \max_{1\le i \le M} s(P_i)$.
Since $\lambda^{s^* - s(P_i)}$ is log-concave, it is a tropical index preserver
by Theorem \ref{thm:TropicalIndexPreserver}.
Hence, we conclude that $\lambda^{s^*} \in \Lambda_d$.

Let us now prove that the recession cone $C$ of $\Ln(\Lambda_d)$ is equal to the set 
log-concave sequences of length $d+1$. 
The fact that the latter set is contained in $C$ follows immediately 
from Theorem \ref{thm:TropicalIndexPreserver}, as each log-concave sequence is a 
tropical index preserver. Conversely, if $\lambda$ is not log-concave,
then the $d$-th truncation $S_\lambda^{\{d\}}$ of its symbol is not tropically real-rooted.
Let $P$ be a tropically real-rooted polynomial, and let $\lambda^*$ be
a log-concave sequence. By a similar argument as above, we can conclude
by letting $s$ tend to infinity, that the tropical polynomial
\[
tr_P^{\lambda^*\lambda^s}(\xi) = \max_{0\le k\le d}(k\xi \ln |a_k| + \ln \lambda_k^* + s \ln \lambda_k)
\]
is not tropically real-rooted. Hence, $\lambda$ is not
contained in the recession cone of the set $\Ln(\Lambda_d)$.

The remaining statements of Theorem~\ref{th:main} follow easily from the above  facts.
\end{proof}

\section{Theorems~\ref{thm:example} and ~\ref{thm:counterexample}}

To settle Theorem~\ref{thm:example}, recall the following statement  proved in e.g., \cite{NoSh}. 
\begin{lm}\label{lem:slopes} For a given real polynomial $P$ and real $x\neq 0$, 
	assume that all tropical roots of $tr_P$ are more than $\log 3$ away from $-\log |x|$. Let $k$ be the tropical index corresponding to $x$. Then $k$ is a central index. In particular, $P(x)\not=0$.
\end{lm}
\begin{proof}
	If $k$ is the tropical index corresponding to $-\log |x|$ then $|a_jx^j|<3^{|k-j|}|a_kx^k|$. Summing  over all $j\not=k$, we get $|a_kx^k|>\sum_{j\not=k}|a_jx^j|$ and the claim follows.
\end{proof}

\begin{cor}\label{cor:sign-ind 2log3}
	Let $P$ be a polynomial of degree $d$ and assume that every integer $k=0,...,d$ is a tropical index of $tr_P$.
	Assume that the tropical roots of $tr_P$ are all simple and more than $2\log 3$ separated one from another. Then $P$ is sign-independently real rooted.
\end{cor}
\begin{proof}
	Indeed, for  $x=\sqrt{a_{k-1}/a_{k+1}}$ the conditions of Lemma~\ref{lem:slopes} are satisfied, so $k$ is a central index and the claim follows from 
    Proposition~\ref{pro:SignIndependently}.\end{proof}

\medskip
Our proof of Theorem~\ref{thm:example} requires two steps. At first, we prove in Lemma~\ref{lem:perturbed binomial positive}  that if a polynomial $P=\dots+a_mx^m+\dots+a_nx^n+\dots$ is  a small perturbation of a polynomial $a_mx^m+\dots+a_nx^n$ with positive coefficients then it has no roots on some positive interval,  with explicit bounds on the dependence of the size of the perturbation on the size of the interval.

Then we  group the tropical roots of $tr_P(\xi)$  
into several clusters of closely located roots and prove that in some neighborhood of each cluster the number of logarithms of positive roots of $P$ is less than or equal to the number of positive tropical roots of  $tr_P$ in this cluster, using a generalization of Rolle's theorem presented in  Lemma~\ref{lem:Rolle}. A similar fact holds  for negative roots as well. 

\begin{lm}\label{lem:perturbed binomial}
	Let  $P$ be a real polynomial and let $U=[\alpha',\alpha'']$ be a real interval such that
	\begin{enumerate}
		\item $tr_P$ has a unique tropical root $\alpha\in U$ corresponding to two monomials $a_mx^m$ and $a_nx^n$, $m<n$, i.e., $\alpha=\frac{\log |a_n|-\log |a_m|}{n-m}$,
		\item $\alpha', \alpha''$ are located more than $\log 4$ away from all tropical roots of $tr_P$,
		\item  for all $l$, $m<l<n$,  
	\begin{equation}\label{eq:perturbed binomial 4}
	\log |a_l|\le v(l)-\log d-\log 4,
	\end{equation}
		where $v(u)=\alpha u+\beta$ is the unique linear function whose graph passes through  $(m,\log |a_m|)$ and $(n,\log|a_n|)$. 
	\end{enumerate}
	Then $P$ has the same number of real roots on the interval $[e^{\alpha'}, e^{\alpha''}]$ as $a_mx^m+a_nx^n$, and the same holds on the interval $[-e^{\alpha''}, -e^{\alpha'}]$.

\end{lm}
\begin{proof}
	The sum $\sum_{k<m}|a_kx^k|$ is less than $\frac{1}{3}|a_mx^m|$ on $\{x\in \bC,\,  \log |x|>\alpha'\}$, compare to the proof of Lemma~\ref{lem:slopes}. Similarly,  $\sum_{k>n}|a_kx^k|\le \frac 1 3 |a_nx^n|$ on $\{x\in\bC,\, \log |x|<\alpha''\}$.
Also, $|\sum_{m<k<n}a_kx^k|\le\frac 1 4 \left( |a_mx^m|+|a_nx^n|\right)$ on $\{x\in \bC,\, \alpha'\le \log |x|\le \alpha''\}$.
	
Consider the case  $I=[e^{\alpha'}, e^{\alpha''}]$; the  case of  $I=[-e^{\alpha''},- e^{\alpha'}]$ is treated similarly. Assume first that $a_nx^n$ and $a_mx^m$ have the same signs on this interval. This means that they together dominate  the sum of all other terms,  and there are no zeros on $I$ at all.
	
If the signs are different, choose a curvilinear rectangle $\Pi$ containing $I$ and bounded by $\{\log|x|=\alpha'\}$, $\{\log|x|=\alpha''\}$ and $\{\arg x=\pm \pi/(n-m)\}$. The inequalities above imply that $a_mx^m$ dominates the sum of all other terms on $\{\log|x|=\alpha'\}$. Similarly, $a_nx^n$ dominates the sum of all other terms on $\{\log|x|=\alpha''\}$.
	
Moreover, the sum  $a_mx^m+a_nx^n$ dominates the sum of all other terms on $\{\log|x|\in U, |\arg x|=\pi/(n-m)\}$ as the arguments of $a_mx^m$ and $a_nx^n$  are equal there. In other words, the increment of the argument of $P$ on the boundary of $\Pi$ is the same as that of $a_mx^m+a_nx^n$. Therefore $P$ has a unique root in $\Pi$, which is necessarily real.
\end{proof}

\begin{cor}
Assume that the tropical roots of $tr_P$ are at least $2\log 4$ apart   from one another. Assume also that for any $l$ lying between two consecutive tropical indices $m,n,$  inequality \eqref{eq:perturbed binomial 4} is satisfied.	
Then the number of positive (resp. negative) roots of $P$ is equal to the number of positive (resp. negative) tropical roots of $P$.
\end{cor}

\medskip
We will need a more refined version of Lemma~\ref{lem:perturbed binomial} to take into account the signs of tropical roots.
\begin{lm}\label{lem:perturbed binomial positive}
	Let  $P$ be a real polynomial and let $m<n$ be its two tropical indices with $a_m, a_n>0$.
	Let $U=[\alpha',\alpha'']$ be a real interval such that
	\begin{enumerate}
		\item the tropical index of any $u\in U$ lies in $[m,n]$ and $U$ is more than $\log 4$ away from the tropical roots of $tr_P$ corresponding to the edges of $\widetilde A_P(u)$ lying outside of $[m,n]$,
		\item  for all $l$, $m<l<n,$ we have that either $a_l>0$ or
\begin{equation}
		\log |a_l|\le v(l)-\log d-\log 4,
		\end{equation}
		where $v(u)=\alpha u+\beta$ is the linear function whose graph passes through  $(m,\log |a_m|)$ and $(n,\log|a_n|)$. 
	\end{enumerate}
	Then $P$ has no roots on $I=[e^{\alpha'}, e^{\alpha''}]$.
\end{lm}

\begin{proof} Let $x\in I$.
 As before, the sum $\sum_{k<m}|a_kx^k|$ is at most  $\frac{1}{3}a_mx^m$ on $I$, as in the proof of Lemma~\ref{lem:slopes}. Similarly,  $\sum_{k>n}|a_kx^k|\le \frac 1 3 a_nx^n$ on $I$.
	Also, $\sum'_{m<k<n}|a_k|x^k\le\frac 1 4 \left( a_mx^m+a_nx^n\right)$ on $I$, where the sum is taken over all  monomials with negative coefficients. Therefore $P>0$ on $I$.
\end{proof}

\subsection{Generalized Rolle's theorem}

For a given nonnegative integer $k$, define the differential operator $L_k$ by 
$$L_k\left(\sum a_jx^j\right):=\sum (j-k)a_jx^j.$$
One can easily check that the latter definition is equivalent to 
$$L_k(P):=x^{k+1}\left(x^{-k}P\right)'.$$

\medskip
 The following variation of Rolle's theorem immediately follows from the second definition of $L_k$.
\begin{lm}\label{lem:Rolle}
	Let $I\subset \bR_+$ be some interval, then 
	$$
	\#\{x\in I, L_k(P(x))=0\}\ge \#\{x\in I, P(x)=0\}-1.
	$$
\end{lm}

\medskip 
One can define a natural  tropical counterpart $l_k$ of $L_k$  as 
$$l_k(\{\epsilon_j\}_{j=0}^n)=\{\sgn(j-k)\epsilon_j\}_{j=0}^n,$$
where $\{\epsilon_j\}_{j=0}^n$ is any sequence of real numbers. 
 Evidently, the number of sign changes in $\{\epsilon_j\}$ differs from that in $l_k(\{\epsilon_j\})$ by at most one.
%\begin{lm}	\label{lem:trop Rolle}
%	The number of positive tropical roots of $L_kP$ is at least the number of positive tropical roots of $P$ minus $1$.
%\end{lm}

\medskip 
Let $\alpha_k$ be the tropical roots of $tr_P$ in the decreasing order. Let $U$ be a connected component of the  $\rho$-neighborhood of $\{\alpha_k\}$,   where $\rho=\log36d$. 

Denote by $[m,n]$ the maximal  interval such that the restriction of $\widetilde A_P$ to this interval has edges with slopes equal to the tropical roots of $tr_P$ lying in $U$.  (We can assume that $n>m+1$ since the case $n=m+1$ is covered by Lemma~\ref{lem:perturbed binomial}.) 

We choose a sequence $\la=\{\lambda_{k,d}\}_{k=0}^d$ such that 
\begin{equation}\label{eq:def of Delta}
\log \left(\lambda_{k-1,d}^{-1}\lambda_{k,d}^2\lambda_{k+1,d}^{-1}\right)=2\Delta_d:=
\frac{d^2}4\log 36d+(d+1)\log d+\log 4,\quad 1\le k\le d-1.
\end{equation}
Let $q_k=(n_k, \log |a_{n_k}|+\log \lambda_{n_k})$, $k=0,\dots, N$,  be the vertices of $A^\la_P$ on the interval $[m,n]$ in increasing order. Note that $n_0=n$, $n_{N}=m$. Let $\alpha_a>\alpha_{a+1}>\dots >\alpha_b$ will be the tropical roots of $tr_P$ lying in $U$.

\medskip 
Let $\Sigma_U=\{\sgn(a_{n_k})\}$ be the sequence of signs of  $a_{n_k}$. Choose  a sequence $\{m_j\}_{j=1}^M$, $m_j\in\{n_k\}_{k=1}^{N-1}$, such that 

\medskip
\noindent
(i) $l_{m_1} \cdots l_{m_M}(\Sigma_U)$ has no sign changes;

\noindent
  (ii)  $M$ is equal to the number of sign changes of $\Sigma_U$. 
  
  \medskip 
  We can assume that $n>m_1>\dots >m_{M-1}\ge m_M>m$.

\begin{prop}\label{prop:roots of Q}
	The polynomial $Q=L_{m_1}\cdots L_{m_M}(P)$ has no roots in $e^U$. 
\end{prop}

\begin{proof}%[Proof of Proposition~\ref{prop:roots of Q}]
	Without loss of generality we can take $a_n>0$. Moreover, by rescaling of $x$  and multiplication of $P$ by a constant, we can assume that $a_n=|a_m|=1$.
	
Let $Q=\sum_{j=0}^{d}b_jx^j$, $b_j=a_j\prod_{k=1}^{M} (j-m_k)$. 	
 We claim that $Q$ satisfies conditions of Lemma~\ref{lem:perturbed binomial positive}.
 
 Let us start with the first condition of Lemma~\ref{lem:perturbed binomial positive}.
 Let $l<m$ and 
 $$
 \kappa^Q_{l,m}=\frac{\log|a_l|+\sum_{k=1}^{M}\log|l-m_k|-\log|a_m|-\sum_{k=1}^{M}\log|m-m_k|}{l-m}
 $$ be the slope of the segment joining the two points in $A_Q$ corresponding to the monomials of degree $l$ and $m$.
 We have 
 \begin{equation}\label{eq:slopes of Q}
 \kappa^Q_{l,m}= \kappa^P_{l,m}-\frac 1 {m-l}\sum_{k=1}^{k_U-1}\log\frac{n_k-l}{n_k-m}.
 \end{equation}
 Elementary computations show that  
 $$
 \frac 1 {m-l}\log\frac{m_k-l}{m_k-m}=\frac 1 {m_k-m}\left(t^{-1}\log(1+t)\right)\le\frac1{m_k-m}, \quad t=\frac{m-l}{m_k-m}>0,
 $$
 as the function $t^{-1}\log(1+t)$ is monotone decreasing.
 
 \medskip
 Therefore the last sum in \eqref{eq:slopes of Q} is bounded from above by $(2+\log d)$; thus   
 $$
 \kappa^Q_{l,m}\ge \alpha_{a-1}-2-\log d,
 $$
 and is  more than $\log 4 $ away from $U$, as  $\rho>2+\log d+\log 4$. Similarly, $\kappa^Q_{l,n}\le \alpha_{b+1}+2+\log d$ for $l>n$.
 This means that all slopes of $\widetilde A_Q$ to the left or to the right of $[m,n]$ are more than $\log 4$ away from $U$ which  shows that the first condition of Lemma~\ref{lem:perturbed binomial positive} is satisfied.

% Here is the idea of proof of the second condition: first, all $b_{n_k}$ are positive, and, second, not too big  as the slopes of $S^1_P$ on $[m,n]$  are not too big by definition of $U$. All other coefficients are therefore small: on the interval  $[n_k,n_{k+1}]$ the function $\log a_l+\log\lambda_l$ is majorated by a linear function by definition of $\{n_k\}$. This implies that
% $$\log a_l\le \Delta (l-n_k)(l-n_{k+1})+\kappa^Q_{n_k,n_{k+1}}+\beta,$$
% where $\kappa^Q_{n_k,n_{k+1}}, \beta$ are chosen in such a way that the quadratic polynomial in the right hand side takes values $\log a_{n_k}, \log a_{n_{k+1}}$ at $n_k,n_{k+1}$. As $\Delta$ is very big and the impact of $\sum\log|l-m_k|$ is small, we get the required bound for $b_l$, $l\not\in\{n_k\}$.

\medskip
To prove the second condition, we use the following elementary statement. 
\begin{lm}\label{lem:quadratic}
	Let $\phi(u)$ be a continuous concave piecewise linear function on $[m,n]$ which is  linear on each segment $[k,k+1]$, $k\in\bZ$; we denote by $\mu_k$ its slope on the latter interval. Assume additionally that $\phi(m)=\phi(n)=0.$  Then, 
	\begin{enumerate}
		\item  if   $0\le m_k-m_{k+1}\le 2C$, then $\phi(u)\le C (m-n)^2/4$;
		\item if  $0\le m_k-m_{k+1}=2\Delta_d$, then $\phi(k)\ge (n-m-1)\Delta_d$ for all $m<k<n, k\in\bZ$.
	\end{enumerate}
\end{lm}

\begin{cor}
\begin{equation}
\log |a_l|\le  \frac{d^2}{4}\log 36d,\; m\le l\le n.
\end{equation}
\end{cor}
\begin{proof}
By definition of $U$, one can apply the first claim of Lemma~\ref{lem:quadratic} to the restriction of $\widetilde A_P$ to the segment $[m,n]$.
\end{proof}
\begin{cor}
	Choose $l\in[m,n]$, $l\in\bZ$ and $l\not\in \{n_k\}$. Then 
	$$
	\log|a_l|\le \frac{d^2}{4}\log 36d-\Delta_d,
	$$
	where $\Delta_d$ is the same as in Theorem~\ref{thm:example}. 
\end{cor}
\begin{proof}
	Condition $l\not\in \{n_k\}$ means that
 $\log|a_l|+\log\lambda_{l,d}< \alpha l +\beta$, where  $\alpha, \beta$ are chosen in such a way that $\alpha m +\beta=\log\lambda_{m,d}$ and $\alpha n +\beta=\log\lambda_{n,d}$. Therefore 
 $$
 \log|a_l|\le -\left(\Theta_d(u)-\alpha u-\beta\right),
 $$
and the bound follows from the second claim of Lemma~\ref{lem:quadratic} applied to $\phi(u)=\Theta_d(u)-\alpha u-\beta$.
\end{proof} 
	
Now, $\log|b_l|=\log|a_l|+\sum\log|m_k-l|\le \log|a_l|+d\log d$. Therefore 
$$
\log |b_l|\le \frac{d^2}{4}\log 36d-\Delta_d+d\log d\le -\log d -4,
$$
which  implies the second condition of Lemma~\ref{lem:perturbed binomial positive}, since both $\log|b_m|,\log|b_n|$ are positive.  
This finishes the proof of Proposition~\ref{prop:roots of Q}.
\end{proof}

\begin{cor}\label{cor:at most M roots}
	Let $M$ be the number of sign changes in $\{a_{n_k}\}$, where $\{{n_k}\}$ are  tropical indices of $tr^\lambda_P$ on the interval $[m,n]$. Then 	$P$ has at most $M$ roots on $e^U$.
\end{cor}
\begin{proof}%[Proof of the Corollary]
	Follows from Proposition~\ref{prop:roots of Q},  and  Lemma~\ref{lem:Rolle}.
\end{proof}

\begin{proof}[Proof of Theorem~\ref{thm:example}]

Applying Corollary~\ref{cor:at most M roots} to each connected component of the   $\log36d$-neighborhood of the set of tropical roots of $tr_P$ (and using Lemma~\ref{lem:slopes} outside of it), we see that the number of positive roots of $P$ does not exceed the number of positive tropical roots of $tr^\lambda_P$.

Changing $P(x)$ to $P(-x)$, we get the same statement for the negative roots. In particular,  we conclude that $\{\lambda_{k,d}\}$ defined in \eqref{eq:def of Delta} is a  real-to-tropical root preserver.
\end{proof}

\medskip

To prove Theorem~\ref{thm:counterexample}, we need an auxiliary statement. 

\begin{lm}\label{lem:R}
	There exists a polynomial $R$ of degree $100$ with $4$ simple negative roots, whose   leading and constant  coefficients are equal to $1$ and the remaining coefficients are non-negative and strictly less than $1$.
\end{lm}

\begin{proof}[Proof of Lemma~\ref{lem:R}]
	Set $Q_1(x)=x+1$ and define $Q_{k+1}(x)=Q_k(x)(x^n+1),\; k=2,3,\dots$, where $n$ is the smallest odd number greater than $\deg Q_k$. Note that 
	\begin{enumerate}
		\item all coefficients of $Q_k$ are either $1$ or $0$, %as these are sparse polynomials,
		\item $Q_k(x)$ is divisible by $(x+1)^k$.
	\end{enumerate}
	
	Take $Q_4(x^5)$ (which has a root of multiplicity 4 at $-1$),  add  some small positive multiple of  $(x+1)^3$ to split of  a simple real root from the $4$-tuple root at $-1$, then add an even smaller positive multiple of  $(x+1)^2$  to split of  another simple root from $-1$, and then add an even smaller multiple of $x+1$ to split of the third simple root. (Note that $Q_4(x^5)$ has no monomials of degree $1,2,3$.) 
	
	The resulting perturbation $\tilde{Q}_4$ has four negative roots, is of degree $100$, has a leading term equal to $1$, the constant term $a_0>1$, and all the remaining coefficients at most $1$. (All of them are equal to either $0$ or $1$ except in degrees $1,2,3$, where they are  small  positive numbers). Define $R= a_0^{-1}\tilde{Q}_4(a_0^{1/100}x) = x^{100}+\dots+1$, with all other coefficients non-negative and smaller than $a_0^{-1/100}$.
\end{proof}

\begin{proof}[Proof of Theorem~\ref{thm:counterexample}]
	Starting with the above polynomial $R$, we construct a polynomial $P$ with $4$ negative roots and with only three tropical roots. Note that 
$$
A_R(u)\le \widetilde A_R(u)\equiv 0\qquad\text{for } 0\le u\le 100,
$$ with equality for $u=0$ and $100$ only.
%Therefore $tr_P(\xi)$, being the  Legendre transform of $A^1_R(u)$, has just one root $\xi=0$. 

Choose $c>0$ in Theorem~\ref{thm:counterexample} such that $A_R(u)\le -cu(100-u)$ for $0\le u\le 100$.  Inequality  \eqref{eq:counterexample} implies that $
\Theta_d(u)$  is almost flat on the interval $[k,k+100]$,  see Remark 7. More exactly, there exists a  linear function $\ell(u)$ such that,  
$$
\Theta_d(u)\le \ell(u)+cu(100-u), \quad k\le u\le k+100,
$$
with equality for $u=k, k+100$. Therefore $A_{x^kR}(u)+\Theta_d(u)\le \ell(u )$ for $0\le u\le 100$, with equality for $u=k, k+100$ (i.e., lies below its chord on $[k,k+100]$). Therefore $A^\lambda_{x^kR}(u)$ is linear, and $tr^\lambda_{x^kR}(\xi)$ has just one tropical root.

Now, choose $\delta>0$  so small  that $P=\delta(x^d+1)+x^kR$ still has $4$ negative simple roots. Then $tr^\lambda_P(\xi)$ has at most $3$ tropical roots, 
since only two extra monomials were added. The latter choice of $P$ settles  Theorem~\ref{thm:counterexample}.
\end{proof}

%\section{Theorem~\ref{thm:example}}

\section{Proposition~\ref{pr:6}   }

We start with some explicit information about $\Lambda_d$ and $\Lambda^+_d$ for small $d$, compare to Theorem~\ref{th:main}. 

\begin{lm}
\label{lm:ExplicitSets} 
\begin{enumerate}
\item For $d=1,$  $\Lambda^+_1 = \Lambda_1 = \bR_+$; 
\item For $d=2,$ $\Lambda^+_2 = \Lambda_2 = \{\lambda\, |\, 4\lambda_1^2\geq \lambda_0\lambda_2\}$.
\end{enumerate}
\end{lm}

\begin{proof}
(1) Note that it is enough to consider only fully supported polynomials $P$. 
Then, by normalization, we can assume that $a_0 = a_1 = 1$. For $d=1$ there is nothing to prove. 

\noindent (2) For $d=2$, consider a polynomial $P(x) = 1 + x + a x^2.$ 
 Then, $P(x)$ has two real roots if and only if  $a \leq \frac14$. 
If $a < 0$, then $tr_P^\dagger(\xi)$ has two essential tropical roots for all $a$.  Thus it suffices to consider only the case $a>0$. We need to compare the above inequality to the condition that  the tropical polynomial
\[
tr^\lambda_P(\xi) = \max\big(\ln \lambda_0,\, \xi + \ln \lambda_1,\, 2\xi+ \ln a + \ln\lambda_2 \big),
\]
has two tropical roots. One can easily check that this happens if and only if $\lambda_1^2\geq  a\lambda_0\lambda_2$.
This inequality holds for all $0\le a \le \frac14$ if and only if $4\lambda_1^2\geq  \lambda_0\lambda_2$.
Clearly, the latter inequality is necessary and sufficient also if we restrict ourselves to polynomials
with positive coefficients.
\end{proof}

\begin{lm}
\label{lem:degree4}
For $d=4$,  $\Lambda_4^+$ contains the set defined by the system of inequalities: 
\begin{equation}
\label{eqn:Degree4}
\begin{cases}
2\lambda_1^2 \geq \lambda_0\lambda_2, \quad
9\lambda_2^2 \geq  4\lambda_1\lambda_3, \quad
2\lambda_3^2 \geq \lambda_2\lambda_4, \\
2(\sqrt[4]{3}-1)\lambda_1^4\geq \sqrt[4]{3}\,\lambda_0^3\lambda_4, \quad
2(\sqrt[4]{3}-1)\lambda_3^4\geq \sqrt[4]{3}\,\lambda_0\lambda_4^3.
\end{cases}
\end{equation}
\end{lm}

\begin{proof}
As we consider only $P$ with positive coefficients, we can without loss of generality  restrict ourselves to the case $a_0 = a_4 = 1$,
i.e. 
\[
P(x) = 1 + a_1 x + a_2 x^2 + a_3x^3 + x^4. 
\]
We compare the appearance of its real roots with  the appearance of tropical roots of  the tropical polynomial 
\[
tr^\lambda_P(\xi) = \max\big(\ln\lambda_0,\, \xi + \ln a_1  + \ln\lambda_1,\, 2\xi + \ln a_2  + \ln\lambda_2,\, 3\xi+\ln a_3  + \ln\lambda_3,\, 4\xi+\ln\lambda_4\big), 
\]
where $\la_0,\dots, \la_4$ are variables. 
For real-rooted polynomials, we obtain the inequalities: 
\[
8\lambda_1^2 \geq  3\lambda_0\lambda_2, \quad
9\lambda_2^2 \geq  4\lambda_1\lambda_3, \quad
8\lambda_3^2 \geq 3\lambda_2\lambda_4.
\]
Let us now consider polynomials  $P(x)$ with exactly two real roots. 
When decreasing $a_1, a_2,$ and $a_3$ simultaneously, 
one can only decrease the
number of essential tropical roots. Therefore 
it suffices to prove the statement for
polynomials $P(x)$ with a real double root  only.  
With our normalization, such a polynomial can be written as
\begin{align*}
P(x) & = (r+x)^2\left(r^{-2}+s x +x^2\right) \\
      & = 1 + \left(2r^{-1} + sr^2\right)x + \left(r^{-2} + 2sr + r^2\right)x^2 + \left(2r + s\right)x^3+ x^4.
\end{align*}
Associated  tropical polynomials are of the form
\begin{align*}
tr_P(\xi) = \max \Big(\ln\lambda_0, \,& \,
\xi + \ln\left(2r^{-1} + sr^2\right) + \ln\lambda_1,\\
& 2\xi + \ln\left(r^{-2} + 2sr + r^2\right)+ \ln\lambda_2,\\
& 3\xi + \ln\left(2r + s\right) + \ln\lambda_3,\quad 4\xi+\ln \lambda_4\Big).
\end{align*}
We will divide our consideration into two cases. If $r \leq 1$, then we will require that the first order term dominates the even order terms at some point. If $r \geq 1$ we will require that the third order term dominates the even order terms at some point.
In the first case, we consider the point
\[
\xi_1 = - \ln(2r^{-1} + s r^2) - \ln\lambda_1+\ln\lambda_0
\]
and obtain the inequalities
\[
\frac{\lambda_1^2}{\lambda_0\lambda_2} \geq \frac{1 + 2sr^3 + r^4}{(2 + s r^3)^2 }
\quad \text{and} \quad
\frac{\lambda_1^4}{\lambda_0^3\lambda_4}\geq \frac{r^4}{(2 + s r^3)^4}.
\]
Since we require the coefficients of $P$ to be positive, 
it is sufficient that these inequalities are valid for all $0<r\le 1$ and $s \geq -\frac{2}{\sqrt[4]{3}}$.
We find that
\begin{align*}
&\sup_{r,s}\frac{1 + 2sr^3 + r^4}{(2 + s r^3)^2} = \sup_r \frac{1}{3-r^4} = \frac12.
\intertext{and that}
&\sup_{r,s}\frac{r}{2+sr^3}= \sup_r \frac{r}{2-\frac{2}{\sqrt[4]{3}} r^3}  = \frac{\sqrt[4]{3}}{2(\sqrt[4]{3}-1)}.\\
\end{align*}
Thus, in  case $r \leq 1$ we obtain the inequalities
\[
2\lambda_1^2 \geq \lambda_0\lambda_2
\quad \text{and} \quad
2(\sqrt[4]{3}-1)\lambda_1^4\geq \sqrt[4]{3}\,\lambda_0^3\lambda_4.
\]
By symmetry,  for $r \geq 1,$ we obtain the inequalities
\[
2\lambda_3^2 \geq \lambda_2\lambda_4
\quad \text{and} \quad
2(\sqrt[4]{3}-1)\lambda_3^4\geq \sqrt[4]{3}\,\lambda_0\lambda_4^3.
\]
Altogether, we derived the system \eqref{eqn:Degree4}.
%All in all, we obtain a system of inequalities giving  \emph{sufficient} conditions
%\[
%\begin{array}{c}
%2\lambda_1^2 \geq \lambda_0\lambda_2, \quad
%9\lambda_2^2 \geq  4\lambda_1\lambda_3, \quad
%2\lambda_3^2 \geq \lambda_2\lambda_4, \\
%2(\sqrt[4]{3}-1)\lambda_1^4\geq \sqrt[4]{3}\,\lambda_0^3\lambda_4, \quad
%2(\sqrt[4]{3}-1)\lambda_3^4\geq \sqrt[4]{3}\,\lambda_0\lambda_4^3.
%\end{array}
%\]
%These inequalities  define a convex set whose recession cone is equal to that of set of real-rooted polynomials. Notice that none of the  inequalities is redundent in the above system.
\end{proof}

%\begin{proof}[Proof of Proposition~\ref{pr:hyp}]
%By Newton's inequalities and the fact that log-concave sequences
%are tropical index preservers (see Theorem \ref{thm:TropicalIndexPreserver}), 
%it is easy to show that it suffices to consider the polynomial
%\[
%P(x) = \sum_{k=0}^d {d\choose k} x^k,
%\]
%which we should compare to a general tropical polynomial
%\[
%tr_P(\xi) = \max_{0\le k\le d}\left(k\,\xi + \ln{d\choose k} + \ln\lambda_k\right),
%\]
%with $\lambda_0,\dots, \lambda_d$ considered as  variables.
%Polynomial $tr_P(\xi)$ has $d$ distinct roots if and only if the  $k$-th term  
%is dominating at the intersection point of the graphs of the $(k-1)$-st and the $(k+1)$-st terms.
%The intersection point in question is given by:
%\[
%\xi_k  =\frac12 \left( \ln\frac{k(k+1)}{(d-k+1)(d-k)}  + \ln\lambda_{k-1} - \ln\lambda_{k+1}\right).
%\]
%Thus we obtain the system of inequalities
%\[
%2\ln\lambda_k \geq \ln\frac{k(d-k)}{(k+1)(d-k+1)} + \ln\lambda_{k-1}+\ln\lambda_{j+1}, \; k=0,\dots, d, 
%\]
%which are obviously fulfilled for  $e^{-k^2}$.% $\lambda_k = {d\choose k}$.
%\end{proof}

\begin{proof}[Proof of Proposition~\ref{pr:6}]
Up to degree $3$, the statement is covered by Lemma~\ref{lm:ExplicitSets}, as there is nothing to
prove in the case of a cubic polynomial with one real root.
The case of degree $4$ follows immediately from Lemma \ref{lem:degree4}.
%\noindent
%Case of degree $5$. Analogously to the previous case, considering  real-rooted polynomials, we obtain the inequalities
%\begin{align*}
%\begin{cases}
%2\ln\lambda_1 & \geq \ln\frac{2}{5}+\alpha_{2},\\
%2\ln\lambda_2 & \geq \ln\frac{1}{2} + \alpha_{1}+\alpha_{3},\\
%2\ln\lambda_3 & \geq \ln\frac{1}{2} + \alpha_{2}+\alpha_{4},\\
%2\alpha_4 & \geq \ln\frac{2}{5} + \alpha_{3}.
%\end{cases}
%\end{align*}
%Let us now consider polynomials with exactly three real roots. It is enough to consider
%polynomials with a double real root. They can be written as
%\begin{align*}
%P(x) & = (u+x)(r+x)^2\left(u^{-1}r^{-2}+s x +x^2\right) \\
%      & = 1 + \left(2r^{-1} +u^{-1}+r^2 s u\right)x + \left(r^{-2}+ r^2 s + 2 r^{-1}u^{-1}+ r^2 u + 2rsu\right)x^2 \\
%      & \quad + \left(r^2 +2rs + r^{-2}u^{-1} + 2ru + su\right)x^3  + \left(2 r + s + u\right) x^4 + x^5.
%\end{align*}
%We we compare with the corresponding tropical polynomial. There are $4$ relevant cases of subsets of
%dominant monomials;
%\[
%\{0,1,2,5\}, \{0,3,4,5\}, \{0,1,4,5\}, \text{ and } \{0,2,3,5\}.
%\]
%\medskip
%\begin{center}
%\underline{Q}: Which set to consider on which assumptions on $u$ and $r$?
%\end{center}
%
%\noindent
%Case of degree $6$. 
\end{proof}

%%%%%%%%%%% 
%
%\section {Tropical discriminant and Manin-Schechtman arrangement}
%

%%%%%%%%%
\section{Application to zero-diminishing sequences}\label{sec:appl}
%%%%%%%%%%

We start with the following standard definition, see e.g., \cite {CC2}, \cite {CC3}.

\begin{defi} A sequence $\Gamma=\{\lambda_k\}_{k=0}^d$ of real numbers is called a 
{\em complex zero decreasing sequence in degree $d$}  {\rm(}a  {\rm CZDS in degree $d$},  for short{\rm)}  if, for any polynomial $P=a_0+a_1x+\dots+a_dx^d$ with real coefficients,  the polynomial $T_\lambda(P)=\lambda_0a_0+\lambda_1a_1x+\dots+\lambda_da_dx^d$ has no more non-real roots than $P$.

 A sequence $\Gamma=\{\lambda_k\}_{k=0}^\infty$ of real numbers is called a 
{\em complex zero decreasing sequence }  {\rm(}a  {\rm CZDS},  for short{\rm)}  if for every $d\in \N$ the sequence $\Gamma=\{\lambda_k\}_{k=0}^d$  is a   CZDS in degree $d$.
% i.e. 
%$Z_\bC(\Gamma(P))\le Z_\bC(P)$, where $Z_\bC(g)$ stands for the number of non-real roots of a polynomial $g$ counting multiplicities.}
\end{defi}
 
Laguerre's classical result from 1884 gives the so far best recipe how to generate such sequences. Namely, 

\begin{tm}[p.~116 of \cite {La}]\label{th:Lag}
For any real polynomial $f(z)$ with all strictly negative roots, the sequence $\{f(n)\},\; n=0,1,\dots  $ 
is a CZDS. 
\end{tm}

On p.\ 382 of his well-known book \cite {Ka}, S.~Karlin posed the problem of characterizing the inverses of CZDS which are called {\em zero-diminishing sequences}  ({\rm ZDS},  for short{\rm}). This problem is sometimes referred to as the Karlin problem.\footnote{ In \cite{CC2} the   authors initially claimed that they have solved Karlin's problem,  but later they discovered a mistake  in the presented solution.} 
Substantial information about CZDS can be found in section 4 of \cite{CC4} and a number of earlier papers.  
Several interesting attempts to find the converse of Laguerre's theorem and to solve the Karlin problem were carried out over the years, the most successful  of them 
apparently being \cite {BCC} and  \cite {BR2}. (For the history of the subject consult  \cite{CC2} and \cite {Pi}.) But inspite  of some hundred 
and thirty years passed since the publication of \cite{La} and  certain  partial progress,  satisfatory  characterization of the sets of all complex zero decreasing sequences and/or  of all zero-diminishing sequences is still unavailable  at  present.  In particular, it is still unknown whether the rapidly decreasing  sequence    $\{ e^{-k^\alpha} \}_{k=0}^\infty$ with $\al>2$ is a  CZDS. 

\medskip 
%Below we use the  initially suggested in \cite{PRS11} idea that one can apply the information about discriminant amoebas obtained  
%in  recent years to attack classical problems in the P\'olya-Schur theory and, in particular, problems related to CZDS/ZDS.  Special 
%features of our method force us to work  with polynomials having only positive coefficients. We need to modify accordingly the  notion of
%CZDS. 
We will now illustrate how the theory developed in this paper can be applied
to obtain new results regarding CZDS.

\begin{tm}
\label{thm:CZDS}
Let $\lambda^* = \{\lambda^*_{k,j}\}_{0\le k\le j, j\in \bN }$ be a triangular 
real-to-tropical root preserver.
Let $\lambda = \{\lambda_k\}_{k=0}^d$ be a sequence of positive numbers.
If the set of central indices of the polynomial
\[
Q_d(x) = \sum_{k=0}^d \frac{\lambda_k}{\lambda^*_{k,d}}\, x^k
\]
is equal to $\{0,1, \dots, d\}$, i.e., $Q_d(x)$ is  sign-independently real rooted, then $\lambda$ is a CZDS in degree $d$.

In particular, if any initial segment $\{\lambda_k\}_{k=0}^d$  of a sequence $\{\lambda_k\}_{k=0}^\infty$ satisfies this condition then $\{\lambda_k\}_{k=0}^\infty$ is a CZDS.
\end{tm}

\begin{proof}
Consider a polynomial $P(x) = \sum_{i=0}^d a_ix^i$, and its image
\[
T_\lambda[P] = \sum_{i=0}^d \lambda_ia_ix^i 
= \sum_{i=0}^d \frac{\lambda_i}{\lambda^*_{i,d}}\, \lambda^*_{i,d}a_ix^i
\]
under the operator $T_\lambda$.
Since $\lambda^*$ is a triangular real-to-tropical root preserver,
the number of essential tropical roots of the polynomial
\[
R(x) =  \sum_{i=0}^d \lambda^*_{i,d}a_ix^i
\]
is at least equal to the number of real roots of $P$.
Let $0=k_0 < k_1 < \dots < k_m = d$ be the tropical indices of
$R(x)$, and let $x_0, \dots, x_m > 0$ be such that the tropical
index $k_j$ is dominating at $x_j$, that is
\begin{equation}
\label{eqn:Ineq1}
\lambda^*_{j,d}|a_j|x_j^j \geq \max_{i\neq j} \lambda^*_{i,d}|a_i|x_j^i.
\end{equation}

Since  each $k_j$ is a central index of the polynomial
$Q_d(x)$, we can find points $y_1, \dots, y_m$ such that 
\begin{equation}
\label{eqn:Ineq2}
\frac{\lambda_j}{\lambda^*_{j,d}}\,y_j^j \geq \sum_{i\neq j}\frac{\lambda_i}{\lambda^*_{i,d}}\,y_j^i.
\end{equation}

Inequalities \eqref{eqn:Ineq1} and \eqref{eqn:Ineq2} imply that
\begin{equation*}
\lambda_j |a_j|(x_jy_j)^j 
=\frac{\lambda_j}{\lambda^*_{j,d}}\,y_j^j\,\lambda^*_{j,d}|a_j|x_j^j 
\geq \sum_{i\neq j}\frac{\lambda_i}{\lambda^*_{i,d}}\,y_j^i\,\lambda^*_{i,d}|a_i|x_j^i
= \sum_{i\neq j}\lambda_i|a_i|(x_jy_j)^i.
\end{equation*}
Thus, each $k_j$ is a central index of $T_\lambda[P]$.
In particular, the number of real roots of $T_\lambda[P]$
is at least equal to the number of essential tropical roots of
$R(x)$, which in turn is at least equal to the number of real roots
of $P$.
\end{proof}

%\begin{tm}
%Assume that Conjecture \ref{conj:main} holds for all degrees $d$.
%Then, the sequence $\{e^{-k^3}\}_{k=0,1,\dots}$ is a CZDS.
%\end{tm}
%
%\begin{proof}
%Using Theorem \ref{thm:CZDS}, it suffices to show that for each $d$, the polynomial
%\[
%Q(x) = \sum_{i=0}^d \frac{e^{-i^3}}{{d \choose i}}\, x^i
%\]
%has all central indices. For $1\le k\le d-1$, let
%\[
%x_k = e^{3k^2 + 1} \sqrt{\frac{(d-k+1)(d-k)}{(k+1)k}}.
%\]
%It is easy to show that $x_{k+1}/x_k > 1$, so that
%$\{x_k\}_{k=1}^{d-1}$ is an increasing sequence of numbers.
%%\[
%%\frac{x_{k+1}}{x_k} = e^{6k+3} \sqrt{\frac{(d-k-1)k}{(d-k+1)(k+2)}}
%%\]
%Similarly, it is easy to show that 
%\[
%\frac{e^{-k^3}}{{d \choose k}}\, x_k^k \geq \max \left(\frac{e^{-(k-1)^3}}{{d \choose k-1}}\, x_k^{k-1}, \frac{e^{-(k+1)^3}}{{d \choose k+1}}\, x_k^{k+1}\right),
%\]
%implying that the monomial with exponent $k$ is dominating in the tropical sense at $x_k$.
%(In particular, $Q(x)$ has all tropical indices.) Therefor, it is enough to show that
%\[
%\frac{e^{-k^3}}{{d \choose k}}\, x_k^k \geq d\max \left(\frac{e^{-(k-1)^3}}{{d \choose k-1}}\, x_k^{k-1}, \frac{e^{-(k+1)^3}}{{d \choose k+1}}\, x_k^{k+1}\right),
%\]
%which is equivalent to that
%\[
%e^{-k^3}{d \choose k-1} x_k \geq e^{-(k-1)^3}d{d \choose k}.
%\]
%This inequality is equivalent to that 
%\[
%e^{6k}  k(d-k)\geq d^2(k+1)(d-k+1)
%\]
%\end{proof}

\begin{tm}
Assume that the sequence $\{e^{-k^2}\}_{k=0}^\infty$ is a real-to-tropical root preserver. 
Then, the sequence $\{e^{-k^\al}\}_{k=0}^\infty$ is a CZDS for all $\al\ge 3$.
\end{tm}

\begin{proof}
For the corresponding polynomial $Q_d(x)=\sum_{k=0}^de^{-k^\alpha+k^2}x^k$ the tropical roots are $\gamma_k=2k-1+(k-1)^\alpha-k^\alpha$. We see that 
$$
\gamma_{k}-\gamma_{k+1}=-2+(k-1)^\alpha+(k+1)^\alpha-2k^\alpha>-2+\alpha(\alpha-1)k^{\alpha-2}$$
as soon as $\al>3$. Already for $\al>2.608\dots$ and $k\ge1,$ the latter expression is bigger than $2\log 3$. Therefore  Corollary~\ref{cor:sign-ind 2log3} implies that $Q_d(x)$ is a sign-independently real rooted for any $\al>3$. Then Theorem~\ref{thm:CZDS} implies the result.
\end{proof}

\begin{rem}{\rm 
The lower bound $\al\ge 3$ for the sequence $\{e^{-k^\al}\}_{k=0}^\infty$ to be a CZDS is apparently not sharp. In particular, computer experiments show that conclusion of Theorem~\ref{thm:CZDS} holds for  $\al> 2.437623\dots$. But since we do not currently see how to prove Conjecture~\ref{conj:main}, we were not trying to get the optional lower bound with the help of  Theorem~\ref{thm:CZDS}.} 
\end{rem}

\end{document}